\documentclass[draft]{amsart}

\usepackage{amssymb}
\usepackage{url}
\usepackage{amscd}

\theoremstyle{plain}
\newtheorem{theorem}{Theorem}

\newtheorem{corollary}[theorem]{Corollary}
\newtheorem{lemma}[theorem]{Lemma}
\newtheorem{proposition}[theorem]{Proposition}

\theoremstyle{remark}

\begin{document}

\date{Received ........}

\title{On a particular integral operator.}

\author{Epaminondas Diamantopoulos}

\address{57 Ag. Eleftheriou, 67100\\
Xanthi\\ Greece}

\email{epdiamantopoulos@yahoo.gr}




\newcommand{\Ds}{\mathbb D}
\newcommand{\norm}[1]{\left\Vert#1\right\Vert}
\newcommand{\abs}[1]{\left\vert#1\right\vert}
\newcommand{\set}[1]{\left\{#1\right\}}
\newcommand{\Real}{\mathbb R}
\newcommand{\eps}{\varepsilon}
\newcommand{\To}{\longrightarrow}
\newcommand{\BX}{\mathbf{B}(X)}
\newcommand{\A}{\mathcal{A}}
\newcommand{\C}{\mathbb C}
\newcommand{\D}{\mathbb D}

\begin{abstract}
We consider an integral operator $\mathcal{I}$, special instances of which was studied in various contexts. Using an appropriate
transformation we write this operator in terms of weighted composition operators. Then, we provide a boundedness criterion on weighted Dirichlet spaces, and we apply this result in order to prove that certain integral operators are bounded on that spaces, unifying this way and extending previous results.
\end{abstract}
\maketitle


\section{Introduction.}
Let $\phi_1$, $\phi_2$, be linear fractional self maps of the unit disc with real coefficients. The linear segment
$S_z=[\phi_1(z), \phi_2(z)]$ is a subset of $\mathbb{D}$ for any $z\in \mathbb{D}$. In this article we consider the integral operator
\begin{equation}\label{int}
\mathcal{I}(f)(z)=\frac{1}{[S_z]}\int_{S_z} \frac{f(\zeta)}{p(z)\zeta +q(z)}\,d\zeta,
\end{equation}
where $p$ and $q$ are meromorphic functions in the unit disc and $[S_z]=\phi_2(z)-\phi_1(z)$, $z\in \mathbb{D}$. Well known instances of the operator $\mathcal{I}$ is the Ces\`{a}ro integral operator
\[
\mathcal{C}(f)(z)=\frac{1}{z}\int_0^z
\frac{f(\zeta)}{1-\zeta}\,d\zeta,
\]
that appears for $\phi_1(z)=0$, $\phi_2(z)=z$, $p(z)=-1$, $q(z)=1$, and the Hilbert integral operator
\[
{\mathcal{H}}(f)(z)=\int_{0}^1\frac{f(\zeta)}{1-\zeta z}\,d\zeta,
\]
that appears for $\phi_1(z)=0$, $\phi_2(z)=1$, $p(z)=-z$ and $q(z)=1$. We study the operator $\mathcal{I}$ on the weighted Dirichlet spaces $\mathcal{D}_{\alpha}$, $0<\alpha<2$, consisting of analytic functions $f$ in the unit disc for which
\[
\|f\|^2_{\mathcal{D}_{\alpha}} = |f(0)|^2+\iint_{\mathbb{D}} |f^{'}(z)|^2(1-|z|)^\alpha\,d m(z).
\]
It can be shown that, for $-1<\alpha<\infty$, $f\in \mathcal{D}_{\alpha}$, if and only if
\[
\sum_{n=1}^\infty n^{1-\alpha}|a_n|^2 <\infty.
\]
From the last expression it comes out that this chain of spaces contains the Hardy space $H^2$, for $\alpha=1$, and the classical Dirichlet space $\mathcal{D}$, for $\alpha=0$.

The purpose of this article is twofold. First, we aim to introduce the operator $\mathcal{I}$ and show that it may be expressed in terms of weighted composition operators, after an appropriate transformation. Our second goal is to exploit the expression of $\mathcal{I}$ in terms of weighted composition operators in order to provide a sufficient boundedness condition  on $\mathcal{D}_{\alpha}$ spaces, $0<\alpha<2$.

The rest of the article is as follows. First, at section 2 we show that under a suitable condition, $\mathcal{I}$ is a well defined operator on every space where a certain growth condition is assumed for every function in that space. In particular, the operator is shown to be well defined when acts on weighted Dirichlet spaces $\mathcal{D}_\alpha$, $0<\alpha<2$, of analytic functions in the unit disc. In section 3 we present various cases of interest where $\mathcal{I}$ naturally appears. In section 4 we provide the promised transformation of $\mathcal{I}$ in terms of a weighted composition operator. In section 5 we prove a rather general sufficient condition for $\mathcal{I}$ to be a bounded operator on $\mathcal{D}_{\alpha}$ spaces, $0<\alpha<2$. Further, in section 6, we specialize this result for simple linear $\phi_1$, $\phi_2$, $p$ and $q$, obtaining this way a convenient condition that is directly applicable for the operators that motivated this study. Finally, in the last section, we provide a procedure that generates well defined instances of the operator $\mathcal{I}$ of arbitrary complexity.

\section{Preliminaries.}
In the sequel, $\mathcal{X}$ will denote a Banach space of analytic functions such that for every function $f\in \mathcal{X}$, and every $z\in \mathbb{D}$, there is a constant $c=c(\mathcal{X})<1$, such that
\begin{align}\label{condition_for_f}
|f(z)| \leq \frac{1}{(1-|z|)^c}\|f\|_{X}.
\end{align}
Examples of a space of the above type is the Hardy space $H^p$, $p>1$, ($c=1/p$, \cite{Du}), the Bergman space $A^p$, $p>2$, ($c=2/p$, \cite{Vu}), and the weighted Dirichlet spaces $\mathcal{D}_{\alpha}$, $0<\alpha<2$, ($c=\alpha/2$, \cite{Ga}).
\begin{lemma} \label{lem : well_defined_integral}
Let $\phi_1$, $\phi_2$, be linear fractional self maps of the unit disc with real coefficients and $p$, $q$ meromorphic functions in the unit disc. If for every $z\in \mathbb{D}$,
\begin{equation}\label{assumption1}
\Re\left[\frac{p(z)\phi_2(z)+q(z)}{p(z)\phi_1(z)+q(z)}\right]^{1/2} >0,
\end{equation}
then, the operator $\mathcal{I}$ is well defined on the space $\mathcal{X}$.
\end{lemma}
\begin{proof}
First assume that $p$ and $q$ are analytic functions in $\mathbb{D}$. Let $r_z(t)=[S_z]t+\phi_1(z)$, $0<t<1$, $z\in \mathbb{D}$. For $f\in \mathcal{X}$, and $z\in \mathbb{D}$,
\begin{align}
|\mathcal{I}(f)(z)|&= \left|\frac{1}{[S_z]}\int_{S_z}\frac{f(\zeta)}{p(z)\zeta+q(z)}\,d\zeta\right|\notag \\
&= \left|\int_0^1\frac{f(r_z(t))}{p(z)r_z(t) +q(z)}\,dt\right|\notag \\
&\leq \int_0^1 \frac{|f(r_z(t))|}{|p(z)r_z(t) +q(z)|} \,dt\notag \\
&\leq \max_{t\in[0,1]}\frac{C}{|p(z)r_z(t) +q(z)|}  \int_0^1 \frac{1}{(1-|r_z(t)|)^{c}}\,dt \|f\|_{\mathcal{X}}.\notag
\end{align}
Now, from the assumption (\ref{assumption1}), after a square power operation, followed by an abstraction of 1, a multiplication by -1 and an inversion, we finish at
\[
-\frac{p(z)\phi_1(z)+q(z)}{p(z)[S_z]}  \notin [0,1],
\]
which in turn implies that the function $p(z)r_z(t)+q(z)$, is a non vanishing function of $t\in [0,1]$, or equivalently
\[
\max_{t\in[0,1]}\frac{1}{|p(z)r_z(t)+q(z)|} < \infty,\quad z\in \mathbb{D}.
\]
Now, since $\phi_i$, $i=1$, $2$, are linear fractional transformations with real coefficients, we get that $\phi_i(\mathbb{D})$, $i=1$, $2$, are symmetrical to the real axis discs, while since $\phi_i$, $i=1$, $2$, are self maps of the unit disc, we get that $-1\leq \phi(\pm1) \leq1$. For any $z\in \mathbb{D}$, and $0<t<1$, we estimate,
\[
|r_z(t)|\leq \min\{[S_{\pm1}]t+\phi_1(-1),[S_{\pm1}]t+\phi_1(1) \},
\]
thus,
\[
1-|r_z(t)|\geq  \max\{1-\phi_1(-1)-[S_{\pm1}]t,1-\phi_1(1)-[S_{\pm1}]t\},
\]
and we get
\[
\int_0^1 \frac{1}{(1-|r_z(t)|)^{c}}\,dt  <\infty,
\]
since $c<1$. The above arguments implies that the operator $\mathcal{I}$ is well defined for any $z\in \mathbb{D}$, and every $f\in \mathcal{X}$. Finally, notice that the assumption of analyticity of $p$ and $q$ may be relaxed. In fact, if either $p$ or $q$ is meromorphic in the unit disc, then a multiplication of both terms of the fraction $f(\zeta)/(p(z)\zeta+q(z))$, with an appropriate polynomial will eliminate their poles, making the above argument valid.
\end{proof}

\section{Major Cases of Interest.}

\subsection{The operator $\mathcal{I}$ as a unification of other integral operators.}
Beside the Ces\`{a}ro and the Hilbert integral operators, the operator $\mathcal{I}$ is a prototype for several other previously studied operators, like the operator $\mathcal{A}$, which is the $H^2$ adjoint operator of the Ces\`{a}ro operator, or the operator $\mathcal{H}_0$, which is the operator that induced by the reduced Hilbert matrix. At Table 1, we provide the choices of $\phi_1$, $\phi_2$, $p$ and $q$ that corresponds at each case, along with some representative articles.
\begin{table}[ht]
\renewcommand{\arraystretch}{1.75}
\caption{Instances of the operator $\mathcal{I}$. }
 \begin{tabular}{lccccccc}
 \hline\hline
\qquad
\text{ Operator } & $\phi_1(z)$ & $\phi_2(z)$ & $p(z)$ & $q(z)$  & Articles \\
\hline
${\mathcal{C}(f)(z)=\frac{1}{z}\int_0^z \frac{f(\zeta)}{1-\zeta}\,d\zeta}$ & $0$ & $z$ & $-1$ & $1$ & \cite{Ga}, \cite{Sis4} \\
${\mathcal{A}(f)(z)=\frac{1}{z-1}\int_1^z f(\zeta)\,d\zeta}$ & $1$ & $z$ & $0$ & $1$ & \cite{Sis2}, \cite{Sis3}\\
${\mathcal{J}(f)(z)=\frac{1}{z-1}\int_1^z \frac{f(\zeta)}{-1-\zeta}\,d\zeta}$  & $1$ & $z$ & $-1$ & $-1$ & \cite{Sis1}\\
${{\mathcal{H}}(f)(z)=\int_{0}^1\frac{f(\zeta)}{1-\zeta z}\,d\zeta}$ & $0$ & $1$ & $-z$ & $1$ & \cite{DS},  \cite{Li}\\
${{\mathcal{H}}_0(f)(z)=\frac{1}{2}\int_{-1}^1\frac{f(\zeta)}{1-\zeta z}\,d\zeta}$ & $-1$ & $1$ & $-z$ & $1$ & \cite{Dia2}\\
\hline
    \end{tabular}
\end{table}

\subsection{Instances of the operator $\mathcal{I}$ emerging as operators induced by matrices.}

The operators $\mathcal{C}$ and ${\mathcal{H}}$ that motivated the study of the integral operator $\mathcal{I}$ are operators induced by the action of particular matrices on coefficients of analytic functions. Naturally, it turns out that we may consider these operators as members of a more general family of matrix - induced operators. In particular, let
\[
M_1 = \left(
  \begin{array}{cccc}
    c_{0,0} & c_{0,1} & c_{0,2} & \ldots \\
    c_{1,0} & c_{1,1} & c_{1,2} & \ldots \\
    c_{2,0} & c_{2,1} & c_{2,2} & \ldots \\
    \vdots & \vdots & \vdots & \ddots \\
  \end{array}
\right),
\]
where
\[
c_{n,k} = (-1)^n\frac{p_0^n}{q_0^{n+1}} \frac{x_2^{n+k+1}-x_1^{n+k+1}}{(x_2-x_1)(n+k+1)},\quad n,k\geq 0,
\]
$-1\leq x_1 < x_2 \leq 1$, $p_0$, $q_0\in \mathbb{R}$, $(q_0 \pm p_0x_2)(q_0 \pm p_0x_1)>0$, and
\[
M_2 = \left(
  \begin{array}{cccc}
    d_{0,0} & 0 & 0 & \ldots \\
    d_{1,0} & d_{1,1} & 0 & \ldots \\
    d_{2,0} & d_{2,1} & d_{2,2} & \ldots \\
    \vdots & \vdots & \vdots & \ddots \\
  \end{array}
\right),
\]
where, $d_{n,k} = 0$, for $0\leq n<k$, and
\[
d_{n,k} = \frac{(-1)^{n-k}}{q_0}\left(\frac{p_0}{q_0}\right)^{n-k} \frac{\lambda_2^{n+1}-\lambda_1^{n+1}}{(\lambda_2-\lambda_1)(n+1)},\quad n\geq k,
\]
where $p_0$, $q_0 \in \mathbb{R}$, $-1\leq \lambda_1 < \lambda_2 \leq 1$ and $(q_0\pm p_0\lambda_1)(q_0 \pm p_0\lambda_2)>0$.

For $p_0 =  - q_0$, the matrices of the family $M_1$ are Hankel matrices while they may be considered as a generalization of the Hilbert matrix, a case that appears for the choice $p_0 = -1$, $q_0 =1$, $x_1=0$, $x_2=1$. The reader may also verify that the reduced Hilbert matrix appears for the  choice $p_0 = -1$, $q_0 =1$, $x_1=-1$, $x_2=1$. On the other hand, the matrices of the family $M_2$ are lower triangular matrices that they may be considered as a generalization of Ces\`{a}ro matrix which appears for $p_0 = -1$, $q_0 =1$, $\lambda_1=0$, $\lambda_2=1$. For an analytic function $f(z)=\sum_{n=0}^{\infty}a_nz^n \in \mathcal{X}$, let
\[
{\mathcal{M}_1} : \sum_{n=0}^{\infty}a_nz^n \to \sum_{n=0}^{\infty}\sum_{k=0}^{\infty} a_k c_{n,k}  z^n,
\]
and
\[
{\mathcal{M}_2} : \sum_{n=0}^{\infty}a_nz^n \to \sum_{n=0}^{\infty}\sum_{k=0}^{n} a_kd_{n,k}  z^n.
\]
Now, assume that for any $f(z) = \sum_{n\geq} a_n z^n$, the above infinite sums converges and define analytic functions in the unit disc for every $f\in \mathcal{X}$. Then,
\begin{align}
\mathcal{M}_1(f)(z)&= \sum_{n=0}^{\infty}\sum_{k=0}^{\infty} a_k(-1)^n\frac{p_0^n}{q_0^{n+1}} \frac{x_2^{n+k+1}-x_1^{n+k+1}}{(x_2-x_1)(n+k+1)}  z^n \notag \\ &=\sum_{n=0}^\infty \left( \sum_{k=0}^\infty a_k (-1)^n \frac{p_0^n}{q_0^{n+1}} \frac{1}{x_2-x_1}\int_{x_1}^{x_2} \zeta^{n+k} \,d\zeta \right) z^n \notag \\ &= \frac{1}{x_2 - x_1} \int_{x_1}^{x_2} \frac{f(\zeta)}{p_0 z \zeta+ q_0} \,d\zeta, \notag
\end{align}
and
\begin{align}
\mathcal{M}_2(f)(z)&= \sum_{n=0}^{\infty}\sum_{k=0}^n  a_k \frac{(-1)^{n-k}}{q_0}\left(\frac{p_0}{q_0}\right)^{n-k} \frac{\lambda_2^{n+1}-\lambda_1^{n+1}}{(\lambda_2-\lambda_1)(n+1)} z^n \notag \\ &=\sum_{n=0}^\infty \left( \sum_{k=0}^n a_k \frac{(-1)^{n-k}}{q_0}\left(\frac{p_0}{q_0}\right)^{n-k} \frac{1}{\lambda_2-\lambda_1}\int_{\lambda_1 z}^{\lambda_2 z} \zeta^{n} \,d\zeta \right) z^n \notag \\ &= \frac{1}{(\lambda_2 - \lambda_1)z} \int_{\lambda_1 z}^{\lambda_2 z} \frac{f(\zeta)}{p_0\zeta +q_0} \,d\zeta. \notag
\end{align}
Notice that the interchange of sum and integral is justified since by the assumed conditions on $p_0$, $q_0$, $x_i$, $\lambda_i$, $i=1$, $2$, and Lemma \ref{lem : well_defined_integral}, the integrals are well defined in both cases and we arrive to our stated goal that the operators $\mathcal{M}_1$ and $\mathcal{M}_2$ are special cases of the operator $\mathcal{I}$.

\section{The operator $\mathcal{I}$ in terms of weighted composition operators.}
We remind that $\mathcal{X}$ is a space of analytic functions such that the condition (\ref{condition_for_f}) holds. For any $(t,z) \in (0,1)\times \mathbb{D}$, let
\[
w(t,z) = \frac{1}{(\phi_2(z)-t[S_z])p(z)+q(z)},
\]
and
\begin{align}
\gamma(t,z)&=\frac{\phi_1(z)\phi_2(z)p(z)+(\phi_1(z)+t[S_z])q(z)}{(\phi_2(z)-t[S_z])p(z)+q(z)}. \notag
\end{align}
Notice that in general the function $\gamma$ is a meromorphic function on the unit disc.
\begin{proposition}\label{prop: represent}
Let $\phi_1$, $\phi_2$, be linear fractional self maps of the unit disc with real coefficients, $p$, $q$ meromorphic in $\mathbb{D}$, such that
\[
\Re \left[\frac{p(z)\phi_2(z)+q(z)}{p(z)\phi_1(z)+q(z)}\right]^{1/2} >0,
\]
and for any $(t,z) \in (0,1)\times \mathbb{D}$,
\begin{equation}\label{eq : gamma self map}
|\phi_1(z)\phi_2(z)p(z)+(\phi_1(z)+t[S_z])q(z)|< |(\phi_2(z)-t[S_z])p(z)+q(z)|.
\end{equation}
Then, for any $f \in \mathcal{X}$,
\[
\mathcal{I}(f)(z) = \int_0^1 T_t(f)(z) \,dt,
\]
where
\[
T_t(f)(z)=w(t,z)f(\gamma(t,z)).
\]
\end{proposition}
\begin{proof}
From the first assumption and Lemma \ref{lem : well_defined_integral} the operator $\mathcal{I}$ is well defined on the space $\mathcal{X}$, while from the second assumption the function $\gamma$ is a well defined self map of the unit disc. We easily verify that $\gamma(0,z)=\phi_1(z)$,  and
$\gamma(1,z)=\phi_2(z)$. At the integral (\ref{int}) we apply the transformation
$\zeta \to \gamma(t,z)$, and we compute,
\begin{align}
\mathcal{I}(f)(z) &=\frac{1}{[S_z]}\int_0^1\frac{f(\gamma(t,z))}{p(z)\gamma(t,z)+q(z)}\frac{\partial\gamma(t,z)}{\partial t}\,dt\notag.
\end{align}
Now, a calculation shows that
\begin{align}
p(z)\gamma(t,z)+q(z)=\frac{[p(z)\phi_1(z)+q(z)][p(z)\phi_2(z)+q(z)]}{p(z)(\phi_2(z)-t[S_z])+q(z)}, \notag
\end{align}
while
\begin{align}
\frac{\partial\gamma(t,z)}{\partial t} &=\frac{[S_z] [p(z)\phi_1(z)+q(z)][p(z)\phi_2(z)+q(z)]}{[p(z)(\phi_2(z)-t[S_z])+q(z)]^2}.  \notag
\end{align}
A substitution gives
\begin{align}
\mathcal{I}(f)(z)&=  \int_0^1 \frac{f(\gamma(t,z))}{p(z)(\phi_2(z)-t[S_z])+q(z)}\,dt \notag \\
&=\int_0^1 w(t,z)f(\gamma(t,z)) \,dt \notag ,
\end{align}
which is the desired result.
\end{proof}
\textbf{Remark :} The assumption (\ref{eq : gamma self map}) is necessary in order to verify that the transformation $\zeta \to \gamma(t,z)$ is a well defined operation on analytic functions. An interesting function theoretic question that arises is whether this assumption may be replaced by a simpler and easily applicable alternative. The author provides an equivalent condition in the last remarks, which is non practical, though the only one available at the moment.

\section{Weighted Dirichlet space norm estimate of the operator $\mathcal{I}$.}
Hereafter, we focus in the case $\mathcal{X} = \mathcal{D}_\alpha$, $0<\alpha<2$. We also remind that for any self map $\omega$ of the unit disc, Schwarz's-Pick inequality is at our disposal
\[
\frac{1-|z|}{1-|\omega(z)|}\leq \frac{1}{|\omega^{'}(z)|}, \quad z \in \mathbb{D}.
\]
\begin{lemma}\label{lem:lemma1}
Let $f\in \mathcal{D}_{\alpha}$,  $0<\alpha<2$, and for any $(t,z) \in (0,1)\times \mathbb{D}$,
\[
|\phi_1(z)\phi_2(z)p(z)+(\phi_1(z)+t[S_z])q(z)|< |(\phi_2(z)-t[S_z])p(z)+q(z)|.
\]
Then, for $0<\alpha<2$,
\[
\|T_t(f)\|^2_{\mathcal{D}_{\alpha}} \leq C\left[ \int_\mathbb{D} \frac{|\partial_z w(t,z)|^2}{|\partial_z \gamma(t,z)|^{\alpha}}  \,dm(z)+\sup_{z\in\mathbb{D}}\frac{|w(t,z)|^2}{|\partial_z\gamma(t,z)|^\alpha}\right]\|f\|^2_{\mathcal{D}_{\alpha}},
\]
for an appropriate constant $C$ independent of $t$.
\end{lemma}
\begin{proof}
First, notice that from our first assumption, $\gamma$ is an analytic function on the unit disc. Further, for the proof we abbreviate $w_t(z)=w(t,z)$, and $\gamma_t(z)=\gamma(t,z)$, thus $\partial_z w(t,z) = w_t^{'}(z)$, and $\partial_z \gamma(t,z) = \gamma_t^{'}(z)$. Let $f\in \mathcal{D}_{\alpha}$, $0<\alpha<2$. We estimate
\begin{align}
\|T_t(f)\|^2_{\mathcal{D}_{\alpha}} &= |T_t(f)(0)|^2 +\int_\mathbb{D}|(w_t(z)f(\gamma_t(z)))^{'}|^2(1-|z|)^\alpha\,dm(z) \notag \\
&\leq |T_t(f)(0)|^2 +  2 \int_\mathbb{D}|w_t(z)^{'}|^2|f(\gamma_t(z))|^2(1-|z|)^\alpha\,dm(z) \notag \\
&\qquad \qquad \qquad +2 \int_\mathbb{D}|w_t(z)|^2|(f(\gamma_t(z)))^{'}|^2(1-|z|)^\alpha\,dm(z)& \notag  \\
& = |T_t(f)(0)|^2 + 2I_1 +2I_2. \notag
\end{align}
Now,
\begin{align}
I_1 & = \int_\mathbb{D}|w_t^{'}(z)|^2|f(\gamma_t(z))|^2(1-|z|)^\alpha\,dm(z) \notag \\
& \leq C \int_\mathbb{D}\frac{(1-|z|)^\alpha|w_t^{'}(z)|^2}{(1-|\gamma_t(z)|)^{\alpha}} \,dm(z)\|f\|^2_{\mathcal{D}_{\alpha}} \notag \\
& \leq C \int_\mathbb{D} \frac{|w_t^{'}(z)|^2}{|\gamma^{'}_t(z)|^{\alpha}}  \,dm(z)\|f\|^2_{\mathcal{D}_{\alpha}}, \notag
\end{align}
while for $I_2$ we estimate,
\begin{align}
I_2 &= \int_\mathbb{D}|w_t(z)|^2|f^{'}(\gamma_t(z))|^2|\gamma^{'}_t(z)|^2(1-|z|)^\alpha\,dm(z) \notag \\
&= \int_\mathbb{D}|w_t(z)|^2|f^{'}(\gamma_t(z))|^2|\gamma^{'}_t(z)|^2(1-|\gamma_t(z)|)^\alpha\frac{(1-|z|)^\alpha}{(1-|\gamma_t(z)|)^\alpha}\,dm(z) \notag \\
&\leq \sup_{z\in\mathbb{D}}\frac{|w_t(z)|^2}{|\gamma_t^{'}(z)|^\alpha} \int_\mathbb{D}|f^{'}(\gamma_t(z))|^2|\gamma^{'}_t(z)|^2(1-|\gamma_t(z)|)^\alpha\,dm(z)\notag \\
&\leq \sup_{z\in\mathbb{D}}\frac{|w_t(z)|^2}{|\gamma_t^{'}(z)|^\alpha} \|f\|^2_{\mathcal{D}_{\alpha}}. \notag
\end{align}

Finally,
\begin{align}
|T_t(f)(0)|^2 &= |w_t(0)|^2|f(\gamma_t(0))|^2 \leq \frac{C |w_t(0)|^2}{\left(1-|\gamma_t(0)|\right)^\alpha}\|f\|^2_{\mathcal{D}_{\alpha}} \notag \\
&\leq \frac{C |w_t(0)|^2}{|\gamma^{'}_t(0)|^{\alpha}}\|f\|^2_{\mathcal{D}_{\alpha}}\leq C\sup_{z\in\mathbb{D}}\frac{|w_t(z)|^2}{|\gamma_t^{'}(z)|^\alpha} \|f\|^2_{\mathcal{D}_{\alpha}}, \notag
\end{align}
and the reader may verify that the above estimates imply the desired result.
\end{proof}
From the Proposition \ref{prop: represent}, the Minkowski's inequality and the Lemma \ref{lem:lemma1} we get
\begin{theorem}\label{the: upper bound general}
Let $\phi_1$, $\phi_2$, be linear fractional self maps of the unit disc with real coefficients, $p$, $q$ meromorphic functions in $\mathbb{D}$, such that
\[
\Re \left[\frac{p(z)\phi_2(z)+q(z)}{p(z)\phi_1(z)+q(z)}\right]^{1/2} >0,
\]
and for any $(t,z) \in (0,1)\times \mathbb{D}$,
\[
|\phi_1(z)\phi_2(z)p(z)+(\phi_1(z)+t[S_z])q(z)|<|(\phi_2(z)-t[S_z])p(z)+q(z)|.
\]
Then, for any $f\in \mathcal{D}_{\alpha}$, $0<\alpha<2$,
\[
\|\mathcal{I}(f)\|_{\mathcal{D}_{\alpha}} \leq C\int_0^1 \left[ \int_\mathbb{D} \frac{|\partial_z w(t,z)|^2}{|\partial_z \gamma(t,z)|^{\alpha}}  \,dm(z)+\sup_{z\in\mathbb{D}}\frac{|w(t,z)|^2}{|\partial_z\gamma(t,z)|^\alpha}\right]^{1/2}\,dt\|f\|_{\mathcal{D}_{\alpha}}.
\]
\end{theorem}

\section{Instances of the operator $\mathcal{I}$, corresponding to a linear fractional map $\gamma$.}\label{section : linear fractional}
Linear fractional maps are of special interest since in that cases the necessary assumption (\ref{eq : gamma self map}) is easily verified. Moreover, all the concrete examples of $\mathcal{I}$ that motivated our study corresponds to linear fractional $\gamma$ (see Table 2). Thus, a natural thing to do is to provide a more applicable boundedness condition for those cases.

\begin{table}[ht]
\renewcommand{\arraystretch}{1.75}
\caption{Examples of operators of the form $\mathcal{I}$ that motivated this study and corresponds to a linear fractional $\gamma$. }
 \begin{tabular}{lcc}
 \hline\hline
\qquad
\text{ Operator }  &$\left[\frac{p(z)\phi_2(z)+q(z)}{p(z)\phi_1(z)+q(z)}\right]^{1/2}$ & $\gamma(t,z)$ \\
\hline
${\mathcal{C}(f)(z)=\frac{1}{z}\int_0^z \frac{f(\zeta)}{1-\zeta}\,d\zeta}$  & $\sqrt{1-z}$  & $\frac{tz}{(t-1)z+1}$ \\
${\mathcal{A}(f)(z)=\frac{1}{z-1}\int_1^z f(\zeta)\,d\zeta}$ & $1$  & $tz+1-t$ \\
${\mathcal{J}(f)(z)=\frac{1}{z-1}\int_1^z \frac{f(\zeta)}{-1-\zeta}\,d\zeta}$  & $\sqrt{\frac{1+z}{2}}$  & $\frac{(-t-1)z+t-1}{(t-1)z-t-1}$ \\
${{\mathcal{H}}(f)(z)=\int_{0}^1\frac{f(\zeta)}{1-\zeta z}\,d\zeta}$ & $\sqrt{1-z}$ & $\frac{t}{(t-1)z+1}$  \\
${{\mathcal{H}}_0(f)(z)=\frac{1}{2}\int_{-1}^1\frac{f(\zeta)}{1-\zeta z}\,d\zeta}$  & $\sqrt{\frac{1-z}{1+z}}$ & $\frac{2z+4t-2}{(4t-2)z+2}$ \\
\hline
    \end{tabular}
\end{table}
For simplicity, we restrict ourselves to simple linear $\phi_1$, $\phi_2$, that is, we assume that $\phi_1(z)=x_1+\lambda_1z$ and $\phi_2(z)=x_2+\lambda_2z$, $x_i$, $\lambda_i\in \mathbb{R}$, $|x_i\pm \lambda_i|\leq 1$, $i=1$, $2$. Since in general, $\phi_i$, $i=1$, $2$, are linear fractional, this is not the most general case, however it is satisfactory for our purposes.

Clearly, in this case,  $0\leq \deg p ,\deg q \leq 1$. Thus, let $p(z)=p_1z+p_0$, and $q(z)=q_1z+q_0$, where $p_i$, $q_i \in \mathbb{R}$, $i=0$, $1$. A reformulation of
\[
\gamma(t,z)=\frac{\phi_1(z)\phi_2(z)p(z)+(\phi_1(z)+t[S_z])q(z)}{(\phi_2(z)-t[S_z])p(z)+q(z)},
\]
in terms of $z$, gives for $t\in(0,1)$,
\[
\gamma(t,z) = \frac{a_3(t)z^3+a_2(t)z^2+a_1(t)z+a_0(t)}{b_2(t)z^2+b_1(t)z+b_0(t)},
\]
where
\begin{align}
&a_3(t)=\lambda_1\lambda_2p_1,\notag\\
&a_2(t)=(\lambda_2-\lambda_1)q_1t+(x_1\lambda_2+x_2\lambda_1)p_1+\lambda_1q_1+\lambda_1\lambda_2p_0,\notag\\
&a_1(t)=[(x_2-x_1)q_1+(\lambda_2-\lambda_1)q_0]t+\notag \\ &\qquad \qquad +(x_1\lambda_2+x_2\lambda_1)p_0 +x_1x_2p_1+x_1q_1+\lambda_1q_0,\notag \\
&a_0(t)=(x_2-x_1)q_0t+x_1x_2p_0+x_1q_0,\notag
\end{align}
and
\begin{align}
&b_2(t)=-(\lambda_2-\lambda_1)p_1t+\lambda_2p_1,\notag \\
&b_1(t)=-[(x_2-x_1)p_1+(\lambda_2-\lambda_1)p_0]t+p_1x_2+q_1+\lambda_2p_0,\notag\\
&b_0(t)=-(x_2-x_1)p_0t+p_0x_2+q_0.\notag
\end{align}
The function $\gamma$ is linear fractional if and only if $a_2(t)=b_2(t)=a_3(t)=0$, $t\in (0,1)$. In fact it suffice to assume that $a_2(t)=b_2(t)=0$, $t\in (0,1)$, since then we easily get $a_3(t)=0$, $t\in (0,1)$. Then,
\[
\gamma(t,z)=\frac{a_1(t)z+a_0(t)}{b_1(t)z+b_0(t)},
\]
and a careful classification shows that $a_2(t)=b_2(t)=0$, $t\in (0,1)$, corresponds to the integral operators of Table 3. For the operators of the Table 3, the corresponding elements $a_i$, $b_i$, $i=0$, $1$, appear in Table 4.
\begin{table}[h]
\caption{Instances of the operator $\mathcal{I}$ for linear fractional $\gamma$.}
\renewcommand{\arraystretch}{1.75}
 \begin{tabular}{ccccccc}
 \hline \hline
\text{Case} & $\lambda_1$ & $\lambda_2$ & $p(z)$ & $q(z)$ & $\mathcal{I}(f)$\\
\hline
$1$ & $0$ & $0$ & $p_1z+p_0$ & $q_1z+q_0$ &  $\frac{1}{x_2-x_1}\int_{x_1}^{x_2}
\frac{f(\zeta)}{p(z)\zeta+q(z)}\,d\zeta$\\
$2$ & $0$ & $\lambda_2$ & $p_0$ & $q_0$  & $\frac{1}{[S_z]}\int_{x_1}^{x_2+\lambda_2 z}
\frac{f(\zeta)}{p_0\zeta+q_0}\,d\zeta$  \\
$3$ & $\lambda_1$ & $\lambda_2$ & $0$ & $q_0$  & $\frac{1}{q_0[S_z]}\int_{x_1+\lambda_1 z}^{x_2+\lambda_2 z}
f(\zeta)\,d\zeta$ \\
\hline
    \end{tabular}
\end{table}

\begin{table}[h]
\caption{The elements $a_i$, $b_i$, $i=0$, $1$, for linear fractional $\gamma$.}
\renewcommand{\arraystretch}{1.75}
 \begin{tabular}{cccc}
 \hline \hline
Element & Case $1$ & Case $2$ \\
\hline
$a_0(t)$ & $q_0 (x_2-x_1)t+x_1(p_0x_2+q_0)$ & $q_0(x_2-x_1)t+x_1(p_0 x_2+q_0) $  \\
$a_1(t)$ & $q_1(x_2-x_1)t+x_1(p_1 x_2 + q_1 )$ & $\lambda_2 q_0 t + \lambda_2 p_0 x_1$     \\
$b_0(t)$ & $-p_0 (x_2 - x_1)t +p_0 x_2+q_0$ & $ -p_0 (x_2 - x_1)t + p_0 x_2+q_0$  \\
$b_1(t)$ & $- p_1 (x_2-x_1)t+ p_1 x_2+ q_1 $ & $- \lambda_2 p_0 t+\lambda_2 p_0 $    \\
\hline \hline
\raisebox{-3ex}{Case 3} &  $a_0(t) =q_0 (x_2-x_1)t+ q_0 x_1$ & $a_1(t) = (\lambda_2-\lambda_1) q_0 t+\lambda_1 q_0 $  \\
                       & $b_0(t)=q_0$ & $b_1(t)=0$ \\
 \hline
    \end{tabular}
\end{table}

For the operators of the Table 3, we will prove a significant simplification of the Theorem \ref{the: upper bound general}. Before stating the result, we remind that (\cite{HKZ}, Theorem 1.7),
\[
\int_\mathbb{D} \frac{1}{|1-z\overline{w}|^{4-2\alpha}}\,dm(z) \leq  \frac{C}{(1-|w|^2)^{2-2\alpha}},  \quad 0<\alpha<1,
\]
\[
\int_\mathbb{D}\frac{1}{|1-z\overline{w}|^{2}}\,dm(z) \leq  C \log \frac{1}{1-|w|^2},
\]
and
\[
\int_\mathbb{D} \frac{1}{|1-z\overline{w}|^{4-2\alpha}}\,dm(z) \leq  C, \quad 1< \alpha <2,
\]
where $C$ is an appropriate positive constant, not necessary the same at each case.

\begin{proposition}\label{prop : simple linear g}
Let  $\phi_1(z)=x_1+\lambda_1z$, $\phi_2(z)=x_2+\lambda_2z$, $x_i$, $\lambda_i \in \mathbb{R}$, $i=1$, $2$, such that $|x_i\pm \lambda_i|\leq 1$, $p(z)=p_1z+p_0$, $q(z)=q_1z+q_0$, where $p_i$, $q_i \in \mathbb{R}$, $i=0$, $1$ such that $b_2(t)=a_2(t)=0$, and for any $z\in \mathbb{D}$,
\[
\Re \left[\frac{p(z)\phi_2(z)+q(z)}{p(z)\phi_1(z)+q(z)}\right]^{1/2}>0.
\]
Moreover, we assume that for any $t\in (0,1)$,
\begin{itemize}
    \item {$|a_0(t)-a_1(t)|\leq |b_0(t)-b_1(t)|$,}
    \item {$|a_0(t)+a_1(t)|\leq |b_1(t)+b_0(t)|$,}
    \item {$|b_1(t)|<|b_0(t)|$.}
\end{itemize}
Let $\delta(t)=a_1(t)b_0(t)-a_0(t)b_1(t)$, and
\[
A_\alpha(t)=\left\{%
\begin{array}{ll}
    |b_1(t)||b_0(t)|^{-\alpha}(|b_0(t)|-|b_1(t)|)^{\alpha-1}, & \hbox{$0<\alpha<1$,} \\
    |b_1(t)||b_0(t)|^{-1}(-\log (1-|b_1(t)||b_0(t)|^{-1}))^{1/2}, & \hbox{$\alpha=1$.} \\
    1, & \hbox{$1<\alpha<2$,} \\
\end{array}%
\right.
\]
If $\delta^{-\alpha/2}A_\alpha \in L^1([0,1])$, then the operator $\mathcal{I}$ is bounded on $\mathcal{D}_\alpha$ space, $0<\alpha<2$.
\end{proposition}
\begin{proof}
From Lemma \ref{lem : well_defined_integral} we get that the operator $\mathcal{I}$ is well defined on spaces $\mathcal{D}_\alpha$, $\alpha \in (0,2)$. Since $b_2(t)=a_2(t)=0$, we also get that $a_3(t)=0$, and the function $\gamma$ is linear fractional, thus the operator $\mathcal{I}$ is one of the operators that appear in Table 3. Since $|-b_0(t)/b_1(t)|>1$, $t\in(0,1)$, the function  $\gamma(t,\cdot)$ is an analytic linear fractional map, thus the image of the unit disc is a disc. Moreover, the coefficients of $\gamma(t,\cdot)$ are real, thus $\overline{\gamma(t,z)}=\gamma(t,\overline{z})$, which implies that $\gamma(t,\mathbb{D})$ is symmetrical to the real axis. Since the image of a connected set under a continuous function is a connected set, and by assumption $|\gamma(t,\pm 1)|\leq 1$, we get $\gamma(t,[-1,1])\subset [-1,1]$.
The above remarks imply that the function $\gamma(t,\cdot)$ is a well defined self map of the unit disc for any $t\in (0,1)$. Now, we calculate
\begin{align}
I= \int_\mathbb{D}\frac{|\partial_z w(t,z)|^2}{|\partial_z \gamma(t,z)|^{\alpha}}   \,dm(z)&=\frac{|b_1(t)|^2|b_0(t)|^{2\alpha-4}}{|\delta(t)|^{\alpha}}\int_\mathbb{D} \frac{\,dm(z)}{\left|1+\frac{b_1(t)}{b_0(t)}z\right|^{4-2\alpha}}.\notag
\end{align}
For $0<\alpha<1$, we continue
\begin{align}
I&\leq \frac{C|b_1(t)|^2|b_0(t)|^{2\alpha-4}}{|\delta(t)|^{\alpha}}\left(1-\left|\frac{b_1(t)}{b_0(t)}\right|^2\right)^{2\alpha-2}. \notag
\end{align}
Analogously, for $\alpha=1$, we get
\begin{align}
I&\leq \frac{C|b_1(t)|^2|b_0(t)|^{-2}}{|\delta(t)|}\log \frac{1}{1-|b_1(t)|^2|b_0(t)|^{-2}}, \notag
\end{align}
while for $1<\alpha <2$,
\begin{align}
I&\leq \frac{C|b_1(t)|^2|b_0(t)|^{2\alpha-4}}{|\delta(t)|^{\alpha}} \leq C |\delta(t)|^{-\alpha} . \notag
\end{align}
Finally,
\begin{align}
\sup_{z\in\mathbb{D}}\frac{|w(t,z)|^2}{|\partial_z \gamma(t,z)|^\alpha} &= |\delta(t)|^{-\alpha}\sup_{z\in\mathbb{D}}|b_1(t)z+b_0(t)|^{2\alpha-2} \notag,
\end{align}
and notice that
\begin{align}
\sup_{z\in\mathbb{D}}|b_1(t)z+b_0(t)|^{2\alpha-2} &\leq\left\{%
\begin{array}{ll}
    (|b_0(t)|-|b_1(t)|)^{2(\alpha-1)}, & \hbox{$0< \alpha<1$,} \\
   C, & \hbox{$1\leq\alpha<2$.} \notag
\end{array}%
\right.
\end{align}
From Lemma \ref{lem:lemma1} and the above estimates we get the desired result.

\end{proof}

\begin{table}[h]
\renewcommand{\arraystretch}{1.75}
\caption{Instances of the operator $\mathcal{I}$. }
 \begin{tabular}{cccccc}
 \hline\hline
\text{ Operator } & $a_1(t)$ &$a_0(t)$ & $b_1(t)$ & $b_0(t)$ &$a_1(t)b_0(t)-a_0(t)b_1(t)$  \\
\hline
$\mathcal{C}$ & $t$ &$0$ & $t-1$ & $1$ & $t$  \\
$\mathcal{A}$ & $t$ &$1-t$ & $0$ & $1$ & $t$  \\
$\mathcal{J}$ & $-t-1$ &$t-1$ & $t-1$ & $-t-1$ & $4t$  \\
$\mathcal{H}$ & $0$ & $t$ & $t-1$ & $1$ & $t(1-t)$  \\
${\mathcal{H}}_0$ & $2$ & $4t-2$ & $4t-2$ & $2$ & $16t(1-t)$  \\
\hline
    \end{tabular}
\end{table}

\begin{corollary}\label{corr : specific}
The operators $\mathcal{C}$, $\mathcal{A}$, $\mathcal{J}$, $\mathcal{H}$ and ${\mathcal{H}}_0$ are bounded on weighted Dirichlet spaces
$\mathcal{D}_\alpha$, $0<\alpha<2$.
\end{corollary}
\begin{proof}
It is easy to verify that the operators under question are well defined on weighted Dirichlet spaces $\mathcal{D}_\alpha$, $0<\alpha<2$, (Table 2). Some additional necessary calculations are presented in Table 5. Using standard technics one may show that the function $\delta^{-\alpha/2} A_\alpha$ as defined in Proposition \ref{prop : simple linear g}, is integrable over $(0,1)$, for all the above cases. The result comes as a corollary of Proposition \ref{prop : simple linear g}.
\end{proof}
\textbf{Remark} : In \cite{Dia2}, we proved that $\|\mathcal{I}\|_{H^2 \to H^2} \leq \int_0^1 \delta^{-1/2}(t) \,dt$, for the operator $\mathcal{I}$ that corresponds to a linear fractional $\gamma$, with the additional assumption that $\gamma(t,1)=1$. Moreover, this estimate proved to be sharp for the operators $\mathcal{C}$, $\mathcal{A}$ (\cite{Sis2}), ${\mathcal{H}}$ (\cite{DJV}), and ${\mathcal{H}}_0$ (\cite{Dia2}). This is a better estimate than the one that we provide in Proposition \ref{prop : simple linear g} for the space $\mathcal{D}_0=H^2$.  It seems that it is not possible to attain this upper norm bound using the method of this article, although the estimate that we prove in Proposition \ref{prop : simple linear g} is sufficient for our purpose to prove boundedness of that operators on weighted Dirichlet spaces.

\section{Generating the operator $\mathcal{I}$.}

Motivated by the assumption (\ref{assumption1}) we can provide a procedure that generates instances of the operator $\mathcal{I}$ of arbitrary complexity. Thus, let $p$ be an analytic function and $\omega$ a self map of the unit disc. We define
\begin{equation}\label{def:q}
q(z) = \left([S_z]\frac{(1-\omega(z))^2}{4\omega(z)}-\phi_1(z)\right)p(z),
\end{equation}
and we calculate
\[
\Re\left[\frac{p(z)\phi_2(z)+q(z)}{p(z)\phi_1(z)+q(z)}\right]^{1/2}= \Re \left(\frac{1+\omega(z)}{1-\omega(z)}\right) >0,
\]
thus, from Lemma \ref{lem : well_defined_integral} the operator $\mathcal{I}$ is well defined on $\mathcal{D}_{\alpha}$, $0<\alpha<2$. Now, a calculation shows that
\[
\gamma(t,z)=\phi_1(z) + [S_z]\psi(t,z),\quad \text{ where } \quad \psi(t,z)=\frac{t(1-\omega(z))^2}{1+(2-4 t) \omega(z)+\omega(z)^2}.
\]
Notice that the function $\gamma$ is independent of the choice of $p$. For an appropriate choice of $\omega$ such that $\gamma(t,\mathbb{D})\subset \mathbb{D}$, $t\in (0,1)$, we may generate operators of arbitrary complexity, able of being represented as integrals of a weighted composition operator.

In Table 6 we demonstrate 6 choices of $\phi_1$, $\phi_2$, $p$ and $\omega$ along with the corresponding function $q$ that was computed from equation (\ref{def:q}). In Table 7, the corresponding map $\gamma$ appears along with the operator $\mathcal{I}$ that generated in each case. We verified that $\gamma$ is indeed a self map of the unit disc for every $t\in (0,1)$ using a computer software. Notice that at Case 5, the function $q$ is not an analytic function, while at Case 6 the function $\phi_2$ is not linear.

\begin{table}[h]
\renewcommand{\arraystretch}{1.75}
\caption{Instances of the operator $\mathcal{I}$ where $\gamma$ is not linear fractional.}
 \begin{tabular}{cccccccc}
 \hline\hline
\text{Case} & $\phi_1(z)$ & $\phi_2(z)$ & $p(z)$ & $\omega(z)$  & $q(z)$ \\
\hline
1 & $-1$ & $1$ & $2z$ & $\frac{z}{3}$ & $3 + 2 z - \frac{z^2}{3}$ \\
2 &  $-1$ & $1$ & $e^z$ & $\frac{z}{3}$ &  $\frac{-e^z (-9 - 6 z + z^2)}{6 z}$\\
3 & $0$ & $z$ & $z^3+z^2-z+1$ & $\frac{z}{3}$ & $-\frac{(z^2-9) (1 - z + z^2 + z^3)}{12} $ \\
4 & $0$ & $z$ & $z^3-z+2$ & $\frac{z^3}{3}$ & $\frac{3 \left(2-z+z^3\right) \left(1-\frac{z^6}{9}\right)}{4 z^2}$\\
5 & $0$ & $z$ & $z$ & $\frac{1}{3}\frac{z-1/2}{1/2z-1}$ & $\frac{z^2 (35 - 32 z + 5 z^2)}{6 (-2 + z) (-2 + 4 z)}$\\
6 & $0$ & $\frac{2z - 1}{z - 2}$ & $z$ & $\frac{z}{4}$ & $-\frac{(-1+2 z) \left(-16+z^2\right)}{16 (-2+z)}$\\
\hline
    \end{tabular}
\end{table}

\begin{table}[h]
\renewcommand{\arraystretch}{1.75}
\caption{Instances of the operator $\mathcal{I}$ where $\gamma$ is not linear fractional.}
 \begin{tabular}{ccc}
 \hline\hline
\text{Case} & $\gamma(t,z)$ & $\mathcal{I}(f)$ \\
\hline
1 & $\frac{-(3 + z)^2 + 2 t (9 + z^2)}{9 + (6 - 12 t) z + z^2}$ & $\frac{3}{2} \int_{-1}^1 \frac{f(\zeta )}{9-z^2+6z(1+\zeta )} \, d\zeta$ \\
2 & $\frac{-(3 + z)^2 + 2 t (9 + z^2)}{9 + (6 - 12 t) z + z^2}$ & $\frac{1}{2} \int_{-1}^1 \frac{6 e^{-z} z f(\zeta)}{9-z^2+6 z (1+\zeta )} \, d\zeta$ \\
3 & $\frac{t (-3 + z)^2}{9 + (6 - 12 t) z + z^2}$ & $\int_0^z \frac{-12 f(\zeta)}{z(1-z+z^2+z^3) \left(-9+z^2-12 \zeta \right)} \, d\zeta$\\
4 & $\frac{ t z (3-z^3)^2}{9+3 (2-4 t) z^3+z^6}$ & $\int_0^z \frac{-12 z^2 f(\zeta)}{z(2-z+z^3) \left(-9+z^6-12 z^2 \zeta \right)} \, d\zeta $ \\
5 & $\frac{t (-5+z)^2 z}{(7-5 z)^2-12 t \left(2-5 z+2 z^2\right)}$ & $\int_0^z \frac{12 \left(2-5 z+2 z^2\right) f(\zeta)}{5 z^5+z^3 (35-60 \zeta )+24 \zeta^3 +8 z^4 (-4+3 \zeta )} \, d\zeta $ \\
6 & $-\frac{t (-4+z)^2 (-1+2 z)}{(-2+z) \left(-16-8 z+16 t z-z^2\right)}$ & $\frac{z-2}{2z-1} \int_0^{\frac{2z-1}{z-2}} \frac{f(\zeta)16 (z-2)}{ (1-2z)(z^2-16)+16z \zeta (z-2) } \, d\zeta $ \\
\hline
    \end{tabular}
\end{table}

\noindent \textbf{Final Remarks } It would be desirable to find a convenient sufficient condition that will imply that the function $\gamma$ is a self map of the unit disc for all $t \in (0,1)$. Towards this direction, we state the following Proposition that provides an answer, which is rather complicated and non practical, though the only one available at the moment. First, let
\begin{align}
&R_1(z)=[S_z]q(z), \quad R_2(z)=\phi_1(z)[p(z)\phi_2(z) +q(z)], \notag \\
&R_3(z)=-[S_z]p(z), \quad R_4(z)=p(z)\phi_2(z) +q(z).\notag
\end{align}
We prove
\begin{proposition}
Let
\begin{align}
&a_z = \left|
  \begin{array}{ccc}
   \frac{R_1(z)}{R_3(z)} & \Im  \frac{R_1(z)}{R_3(z)} & 1 \\
   \phi_1(z) &  \Im \phi_1(z)  & 1 \\
   \phi_2(z) & \Im \phi_2(z)  & 1 \\
  \end{array}
\right|, \quad d_z =  \left|
  \begin{array}{ccc}
  \left| \frac{R_1(z)}{R_3(z)}\right|^2 & -i \frac{R_1(z)}{R_3(z)} & 1 \\
    |\phi_1(z)|^2 & -i\phi_1(z)  & 1 \\
    |\phi_2(z)|^2 & -i\phi_2(z)  & 1 \\
  \end{array}
\right| \notag \\
&f_z=- \left|
  \begin{array}{ccc}
    \left| \frac{R_1(z)}{R_3(z)}\right|^2 & \Re  \frac{R_1(z)}{R_3(z)} & \Im   \frac{R_1(z)}{R_3(z)}\\
    |\phi_1(z)|^2 &  \Re \phi_1(z)  & \Im \phi_1(z)  \\
    |\phi_2(z)|^2 & \Re \phi_2(z) &  \Im \phi_2(z)  \\
  \end{array}
\right|, \notag
\end{align}
and
\[
K_z = \frac{d_z}{2a_z}, \quad r_z = \sqrt{\frac{d_z^2}{4a_z^2}-\frac{f_z}{a_z}},
\]
Then, $\gamma:(0,1)\times \mathbb{D} \to \mathbb{C}$, is a well defined self map of the unit disc if and only if
\begin{align}
\{ w \in \mathbb{C} : \arg{w} \in (\arg\gamma(0,z), \arg\gamma(1,z),\quad \exists z\in \mathbb{D},\quad |w - K_z|= r_z \}\subseteq \mathbb{D}.\notag
\end{align}
\end{proposition}
\begin{proof}
Notice that the function $\gamma$ is a linear fractional transformation of the variable $t$, that is
\[
\gamma(t,z) = \frac{R_1(z)t+R_2(z)}{R_3(z)t+R_4(z)}.
\]
For any $z\in \mathbb{D}$, the curve $\gamma([0,1],z)$ is a circular arc, laying on the circle that connect the points $\gamma(0,z)$, $\gamma(1,z)$ and $\gamma(\infty,z)$. The transformation $\gamma$ is a well defined self map of the unit disc if and only if for any $z\in \mathbb{D}$, $\gamma([0,1],z) \subset \mathbb{D}$. From any standard text on linear fractional maps (or try the web : [We]) one may verify that the present Lemma is the appropriate analytical formulation of the latter geometrical statement.
\end{proof}
Supplementary to the above, we would like to emphasize the remarkable pattern of the coefficients of the nominator and the denominator of $\gamma$ when expressed as polynomials of $z$. In particular, for the case of simple linear $\phi_i(z)=x_i+\lambda_i z$, $i=1$, $2$, let $p(z)=\sum_{n=0}^N p_n z^n$, and $q(z)=\sum_{n=0}^N q_n z^n$, be complex polynomials. The nominator of $\gamma$ is a polynomial of degree $N+2$ while the denominator is a polynomial of degree $N+1$. A reformulation of $\gamma(t,z)$ in terms of $z$ reveals that
\[
\gamma(t,z) = \frac{\sum_{n=0}^{N+2} a_n(t) z^n}{\sum_{n=0}^{N+1} b_n(t) z^n},
\]
where,
\begin{align}
[a(t)]
=(&(x_2-x_1)[q] +(\lambda_2-\lambda_1)[q]^{'})t+ \notag \\ &+x_1(x_2 [p]+[q])+(x_1\lambda_2+x_2\lambda_1)[p]^{'}+\lambda_1[q]^{'}+\lambda_1\lambda_2 [p]^{''}, \notag
\end{align}
and
\begin{equation}
[b(t)]
=-((x_2-x_1)[p] +(\lambda_2-\lambda_1)[p]^{'})t+x_2[p] +[q]+ \lambda_2 [p]^{'}, \notag
\end{equation}
where
\begin{align}
&[a(t)] = [a_0(t),a_1(t), \ldots, a_{N+1}(t), a_{N+2}(t)], \notag \\
&[b(t)] = [b_0(t),b_1(t), \ldots, b_{N+1}(t), 0], \notag \\
&[p] = [p_0, p_1, \ldots, p_n, 0,0], \notag \\
&[p]^{'} = [0,p_0, p_1, \ldots, p_n,0], \notag \\
&[p]^{''} = [0,0,p_0, p_1, \ldots, p_n]. \notag
\end{align}
Finally, notice that the last calculation clearly extends for $N \to \infty$, that is, for any analytic functions $p$ and $q$ in the unit disc. The reader may also verify that the special case of the above calculations for linear $p$ and $q$ was presented at Section \ref{section : linear fractional}.

\end{document}